\documentclass[runningheads]{llncs}

\usepackage[T1]{fontenc}
\def\doi#1{\href{https://doi.org/\detokenize{#1}}{\url{https://doi.org/\detokenize{#1}}}}
\usepackage{graphicx}
\usepackage{hyperref}
\usepackage{amsmath}
\usepackage{amssymb}
\usepackage{mathtools}
\usepackage{xcolor}
\usepackage{caption}
\usepackage{subcaption}
\usepackage{listings}
\newcommand{\bmat}[1]{\textbf{#1}}
\newcommand{\mvec}[1]{\boldsymbol{#1}}
\DeclareMathOperator\supp{supp}
\newcommand{\GG}{\mathcal{G}}
\newcommand{\SSS}{\mathcal{S}}
\newcommand{\RR}{\mathcal{R}}
\newcommand{\B}{\{0,1\}}

\newcommand{\CC}{\mathcal{C}}

\makeatletter
\def\blfootnote{\gdef\@thefnmark{}\@footnotetext}
\makeatother

\begin{document}
\title{On Minimally Non-Firm Binary Matrices}
\author{R\'{e}ka \'{A}gnes Kov\'{a}cs}%\orcidID{0000-0002-5977-6554}}
\authorrunning{R. A. Kovacs}
\institute{University of Oxford \& The Alan Turing Institute\\
\email{reka.kovacs@maths.ox.ac.uk}}
\maketitle              % typeset the header of the contribution
\begin{abstract}
%The abstract should briefly summarize the contents of the paper in

For a binary matrix $\bmat{X}$, the Boolean rank $br(\bmat{X})$ is the smallest integer $k$ for which $\bmat{X}$ equals the Boolean sum of $k$ rank-$1$ binary matrices, and the isolation number $i(\bmat{X})$ is the maximum number of $1$s no two of which are in a same row, column and a $2\times 2$ submatrix of all $1$s. 
In this paper, we continue Lubiw's study of firm matrices. $\bmat{X}$ is said to be firm if $i(\bmat{X})=br(\bmat{X})$ and this equality holds for all its submatrices.
We show that the stronger concept of superfirmness of $\bmat{X}$ is equivalent to having no odd holes in the rectangle cover graph of $\bmat{X}$, the graph in which $br(\bmat{X})$ and $i(\bmat{X})$ translate to the clique cover and the independence number, respectively.
A binary matrix is \emph{minimally non-firm} if it is not firm but all of its proper submatrices are.
We introduce two matrix operations that lead to generalised binary matrices and use these operations to derive four infinite classes of minimally non-firm matrices.
We hope that our work may pave the way towards a complete characterisation of firm matrices via forbidden submatrices.

%Our main tool of investigation is through the rectangle cover graph of $\bmat{X}$ in which $br(\bmat{X})$ and $i(\bmat{X})$ translate to the clique cover and the independence numbers, respectively.

\keywords{Boolean rank  \and Rectangle covering number \and Firm matrices}
\end{abstract}
\section{Introduction}
\blfootnote{To appear in the Proceedings of ISCO 2022.}
The \emph{Boolean rank} of a binary matrix $\bmat{X}$, $br(\bmat{X})$, is the smallest integer $k$ for which $\bmat{X}$ equals the sum of $k$ rank-$1$ binary matrices, using Boolean arithmetic in which $1+1=1$ holds \cite{Kim:1982}.
A \textit{rectangle} of $\bmat{X}$ is a submatrix of all $1$s. 
Note that the support of a rank-$1$ binary matrix is precisely a rectangle, hence $br(\bmat{X})$  is the minimum number of rectangles needed to cover $\supp(\bmat{X}):=\{(i,j): x_{i,j}=1 \}$.

An \emph{isolated set} of $\bmat{X}$ is a set
$S\subseteq \supp(\bmat{X})$  such that for any distinct $(i_1,j_1)$, $ (i_2,j_2)$ in $S$, it holds $i_1\not = i_2$, $j_1\not = j_2$ and $x_{i_1,j_2}=0$ or $x_{i_2,j_1}=0$. 
The \emph{isolation number} of $\bmat{X}$,
$i(\bmat{X})$, is the maximum cardinality of an isolated set \cite{Gregory:1983}.  %\cite{Monson:1995}. 
In the field of communication complexity, quantities $br(\bmat{X})$ and $i(\bmat{X})$ are often referred to as the rectangle covering number and the fooling set bound \cite{Kushilevitz:1996:CC:264772}.

In the bipartite graph whose biadjacency matrix is $\bmat{X}$, $br(\bmat{X})$ is the minimum number of bicliques (complete bipartite subgraphs) needed to cover the edge set, while $i(\bmat{X})$ is the maximum cardinality of a matching in which no two edges are in a $4$-cycle.
Both  $br(\bmat{X})$ and $i(\bmat{X})$ are NP-hard to compute for general binary \cite{Orlin:1977,Pulleyblank:1982} and totally balanced matrices as well \cite{Muller:1990:AlternatingCycleFreeMatch,Muller:1996:EdgePerfect}.

For any binary matrix $\bmat{X}$, it can be readily checked that $i(\bmat{X}) \le br(\bmat{X})$.
This inequality may however be strict for many matrices.  In fact, the complement of the identity matrix 
%of size $n$, 
%$\overline{\bmat{I}}_n$ which has $0$s on the diagonal and $1$s everywhere else, has $br(\overline{\bmat{I}}_n)=\mathcal{O}(\log_2(n)) $ and $i(\overline{\bmat{I}}_n) = 3$ for $n\ge 3$ which 
shows that the gap between $i(\bmat{X})$ and $br(\bmat{X})$ may be arbitrarily large \cite{Caen:1981}.
%$\overline{\bmat{I}}_n:= \bmat{J}_n - \bmat{I}_n$, where $\bmat{I}_n$ is the  identity and $\bmat{J}_n$ is the all $1$s matrix of dimension $n\times n$, 
%
We say $\bmat{X}$ is \emph{firm} if $i(\bmat{X})=br(\bmat{X})$ and this equality also holds for all its submatrices.
The concept of firmness, along with many results that form the basis of this paper were introduced by Lubiw in \cite{Lubiw:1990}. 
A key tool in Lubiw's work is to define the \emph{rectangle cover graph} of $\bmat{X}$ (the $1$'s graph in her words) in which $i(\bmat{X})$ and $br(\bmat{X})$ translate to the independence and clique cover number, respectively. 
Lubiw defines $\bmat{X}$ to be \emph{superfirm} if $\bmat{X}$'s rectangle cover graph is perfect and demonstrates that superfirm matrices are a strict subset of firm matrices.
In addition, she shows that covering rectilinear polygons by a minimum number of continuous rectangles is a special case of the rectangle cover problem on binary matrices \cite{Lubiw:1990}. 
In the bipartite setting, firmness is later redefined under the name `edge-perfection' \cite{Muller:1996:EdgePerfect}, while superfirmness is investigated under the name `cross-perfection' from a polyhedral perspective \cite{Dawande:2003:NCB:958867.958868}. 
The following important classes of matrices have been shown to be firm.
\emph{Interval} matrices, matrices whose columns can be permuted so the $1$s appear consecutively in each row, are proved to be firm by a deep result of Gy\H{o}ri \cite{Gyori:1984:Interval}.
\emph{Linear} matrices, matrices that have no $2\times 2$ submatrix of $1$s, 
and matrices that can be decomposed into linear matrices via the matrix equivalent of split decomposition on bipartite graphs 
%and node-to-node adjacency matrices of bipartite distance hereditary graphs 
are shown to be superfirm by Lubiw \cite{Lubiw:1990}. 
The firmness of biadjacency matrices of domino-free bipartite graphs is implied by a result of Amilhastre et al. \cite{Amilhastre:1998}.

%\ooo{ polymonio, Chung, Dawande, matrix perfection }

In this paper, we start the investigation of minimally non-firm matrices. A binary matrix $\bmat{X}$ is \emph{minimally non-firm} if $i(\bmat{X})<br(\bmat{X})$ and $i(\bmat{X}')=br(\bmat{X}')$ for all proper submatrices $\bmat{X}'$ of $\bmat{X}$.
Our main tool is looking at the problem through the rectangle cover graph.
%of a binary matrix that was introduced by Lubiw, in which $i(\bmat{X})$ and $br(\bmat{X})$ translate to the independence and clique cover number. 
%Although the perfection of the rectangle cover graph is a stronger requirement than firmness, called superfirmness, it seems fundamental to understand how odd holes and odd antiholes can appear in the rectangle cover graph.
First, we extend a theorem of Lubiw and show that interestingly odd antiholes cannot appear without odd holes in rectangle cover graphs.
Then we characterise the necessary and sufficient submatrices for $5$-holes to appear.
We define \emph{simplicial} $1$s and a procedure for their removal which leads to \emph{generalised} binary matrices. 
We introduce the \emph{stretching} matrix operation  which then along with the simplicial $1$ removal procedure are used to give a general recipe for the construction of minimally non-firm matrices. We then prove by using this general recipe that four infinite classes of matrices are  minimally non-firm. %, two of which are a class of totally balanced matrices.
To the best of our knowledge, minimally non-firm matrices have not been studied before.
%
%To get a better understanding of firmness, one might try to list the minimal violators of this property.
We believe that studying them is a natural approach to better understand firmness, 
akin to the study of perfect graphs via minimally imperfect graphs. % and ideal matrices via minimally non-ideal matrices.
%how minimally imperfect graphs have also been studied 
%\cite{Chudnovsky:2006:SPGT}.
%
%and we are also partly motivated by 
%Our motivation to study them partly stems from
%the success of studying perfect graphs via minimally imperfect graphs.
%
%The study of minimally imperfect graphs lead to 
%just as studying perfect graphs via minimally imperfect graphs.
%
We hope that our results may pave the way towards a complete characterisation of firm and superfirm matrices via forbidden submatrices.

This paper is organised as follows. Section \ref{section_preliminaries} gives a brief recap on the work of Lubiw introducing the concept of rectangle cover graphs, superfirmness and generalised binary matrices.
In Section \ref{section_simplicial_and_stretching}, simplicial $1$s and the stretching operation are introduced.
In Section \ref{section_theorem_odd_antihole}, we show that a matrix is superfirm if and only if it has no odd holes in its rectangle cover graph.
In Section \ref{section_mnf} we prove our main theorem, which we then use to derive four infinite classes of minimally non-firm binary matrices. 
%a theorem which provides a general recipe for the construction of minimally non-firm matrices and show that the conditions of this theorem hold for four infinite classes of matrices. 
We conclude in Section \ref{section_conclusion} and mention two open problems.
%Finally in Section \ref{section_conclusion}, we mention two open problems.

\section{Preliminaries}
\label{section_preliminaries}

Let $\bmat{X}\in \B^{m\times n}$. % and $\supp(\bmat{X}) := \{(i,j): x_{ij}=1 \}$.
For  $I\subseteq [n]:=\{1,\dots,n \}$ and $J\subseteq [m]$, a \emph{submatrix} of $\bmat{X}$ identified by $I\times J$ is obtained by deleting the rows  not in $I$ and the columns not in $J$.  If $I\subsetneq [n]$ or $J\subsetneq[m]$ then $I\times J$ is a \emph{proper submatrix} of $\bmat{X}$.
A submatrix is a \emph{rectangle} if $I\times J\subseteq \supp(\bmat{X})= \{(i,j): x_{i,j}=1 \}$. 
%Note that the support of a rank-$1$ binary matrix is precisely a rectangle, hence $br(\bmat{X})$  is the minimum number of rectangles needed to cover $\supp(\bmat{X})$.
%To a rectangle $I\times J$ of $\bmat{X}$ one may associate a \emph{rectangle matrix} $\mvec{a}\mvec{b}^\top\in \B^{m\times n}$ with $a_i=b_j=1$ for $i\in I, j\in J$ and $0$ otherwise. This sheds light on an alternative equivalent definition of the Boolean rank, which says that $br(\bmat{X})$ is the minimum number of rank-$1$ binary matrices whose Boolean combination is $\bmat{X}$.
As the $1$s in a row or column form a rectangle, we have $br(\bmat{X})\le \min\{m,n\}$. In addition, note that $br(\bmat{X})$ is invariant under transposition and under duplicating rows and columns.

%An \emph{isolated set} of $\bmat{X}$ is a set
%$S\subseteq \supp(\bmat{X})$  such that for any distinct $(i_1,j_1)$, $ (i_2,j_2)$ in $S$, it holds $i_1\not = i_2$, $j_1\not = j_2$ and $(i_1,j_2)\not \in \supp(\bmat{X})$ or $(i_2,j_1)\not \in \supp(\bmat{X})$. 
%The isolation number of $\bmat{X}$,
%$i(\bmat{X})$, is the maximum cardinality of an isolated set.
For an isolated set $S$ and rectangle $I\times J$, we have $|S \cap (I\times J)|\le 1$, hence $i(\bmat{X})\le br(\bmat{X})$. 
Recall that $\bmat{X}$ is \emph{firm} if $i(\bmat{X}')=br(\bmat{X}')$ holds for all submatrices $\bmat{X}'$ of $\bmat{X}$, including $\bmat{X}$.
The \emph{rectangle cover graph} $\GG(\bmat{X})$ of $\bmat{X}$ is the graph on vertex set $\supp(\bmat{X})$, where two vertices are adjacent if they can be covered by a common rectangle of $\bmat{X}$. 
We adopt the convention that vertices of $\GG(\bmat{X})$ are drawn in the positions of the corresponding $1$s'  of $\bmat{X}$.
See Figure \ref{figure_D4} for an example of $\GG(\bmat{X})$ for matrix $\bmat{D}_4$. 
Clearly, the independent sets of $\GG(\bmat{X})$ are just the isolated sets of $\bmat{X}$. 
%\ooo{maximal}
Lubiw shows that maximal cliques of $\GG(\bmat{X})$ are in direct correspondence with maximal rectangles of $\bmat{X}$ \cite{Lubiw:1990}. Therefore we have $i(\bmat{X}) = \alpha(\GG(\bmat{X}))$ and $br(\bmat{X})=\theta(\GG(\bmat{X}))$, where $\alpha(G)$ and $\theta(G)$ denote the independence and clique cover number of a graph $G$, respectively.
A graph $G$ is \emph{perfect} if $\alpha(H)=\theta(H)$ holds for every induced subgraph $H$ of $G$.
A \emph{hole} is an induced chordless cycle of length at least four. An \emph{odd hole} is a hole of odd length and an \emph{odd antihole} is the complement of an odd hole. 
Perfect graphs are exactly those that have no odd holes and no odd antiholes by the Strong Perfect Graph Theorem \cite{Chudnovsky:2006:SPGT}.
$\bmat{X}$ is said to be \emph{superfirm} if $\GG(\bmat{X})$ is perfect \cite{Lubiw:1990}. 
Superfirm matrices are a strict subset of firm matrices \cite{Lubiw:1990}, as for instance $\bmat{D}_4$ is an interval matrix hence firm by Gy\H{o}ri's Theorem \cite{Gyori:1984:Interval} but not superfirm as $\GG(\bmat{D}_4)$ contains a $5$-hole as shown in Figure \ref{figure_D4}. Note that this is because not every induced subgraph of $\GG(\bmat{X})$ corresponds to a submatrix of $\bmat{X}$ and firmness requires $\alpha(H)=\theta(H)$ to hold for only those subgraphs $H$ of $\GG(\bmat{X})$ where $H=\GG(\bmat{X}')$ for a submatrix $\bmat{X}'$ of $\bmat{X}$.
\begin{figure}[tb!]
	\centering
	\begin{minipage}[c]{.3\textwidth}
		\centering
		$$\bmat{D}_4=
		\begin{bmatrix}
		0 & 1 &   1 & 0\\
		1 & 1 &   1 &0  \\
		1& 1&1 & 1   \\
		0 & 0 & 1 & 1
		\end{bmatrix}$$
		%		\subcaption{\label{figure_H3_matrix}Matrix Structure}
	\end{minipage}
	\begin{minipage}[c]{.3\textwidth}
		\centering
		\includegraphics[scale=0.7]{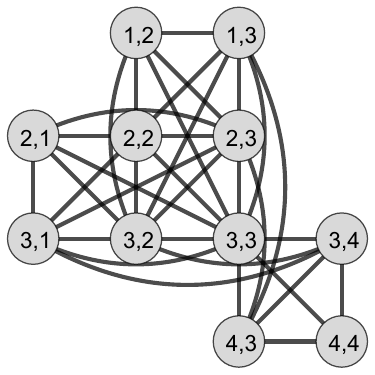}
	\end{minipage}
	\begin{minipage}[c]{.3\textwidth}
		\centering
		\includegraphics[scale=0.68]{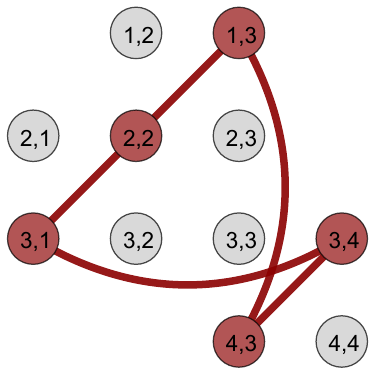}
	\end{minipage}
	\caption{\label{figure_D4}$\bmat{D}_4$, its rectangle cover graph $\GG(\bmat{D}_4)$ and  the $5$-hole in $\GG(\bmat{D}_4)$ highlighted}
\end{figure}

Replacing a $1$ of $\bmat{X}$ at $(i,j)$ with a $0$ does not necessarily correspond to the deletion of vertex $(i,j)$ from $\GG(\bmat{X})$ as edges not incident to $(i,j)$ may get deleted. 
To represent all induced subgraphs of $\GG(\bmat{X})$ in matrix form, Lubiw introduces a new entry type $?$ which may be part of a rectangle but need not be covered in a feasible covering.
A matrix over $\{0,1, ? \}$ is called a \emph{generalised binary matrix} \cite{Lubiw:1990}. 
A \emph{rectangle} of a generalised binary matrix $\bmat{Y}$ is a submatrix containing no $0$s, while an \emph{isolated set} of $\bmat{Y}$ is a subset of $\supp(\bmat{Y}):= \{(i,j): y_{i,j}=1 \}$ in which no two elements are contained in a common rectangle of $\bmat{Y}$. 
Then $i(\bmat{Y})$, $br(\bmat{Y})$ and firmness are analogously defined as for standard binary matrices.
For $\bmat{X}\in\B^{m\times n}$ and $P\subseteq \supp(\bmat{X})$, let $\bmat{X}^{P}$ be the generalised binary matrix obtained from $\bmat{X}$ by replacing all $1$s in $P$ by $?$s, i.e. $x_{i,j}^P=?$ for $(i,j)\in P$ and $x_{i,j}^P=x_{i,j}$ otherwise.
%If $P=\{(\ell,k)\}$ then we will use the notation $\bmat{X}^{(\ell,k)}$.
For $\bmat{X}^P$ define its rectangle cover graph $\GG(\bmat{X}^P)$ to be the subgraph of $\GG(\bmat{X})$ induced by $\supp(\bmat{X})\setminus P$.
%One then can see that
Superfirmness of $\bmat{X}$ is then equivalent to the requirement that $i(\bmat{X}^P)=br(\bmat{X}^P)$ for all $P\subseteq \supp(\bmat{X})$ \cite{Lubiw:1990}.

\section{Simplicial $1$s and Stretching}
\label{section_simplicial_and_stretching}

%A \emph{simplicial vertex} of a graph is one whose neighbours form a clique. 
Let $\bmat{Y}$ be a generalised binary matrix.
We say $(\ell,k)\in\supp(\bmat{Y})$ is a \emph{simplicial $1$} of $\bmat{Y}$ if $I\times J$ with $I=\{i : y_{i,k}\in \{1,?\}  \}$ and  $J=\{j : y_{\ell, j}\in \{1,?\} \}$ satisfies $I\times J\subseteq \{(i,j) : y_{i,j}\in \{1,? \} \}$, that is $I\times J$ is a rectangle of $\bmat{Y}$. Note that 
%for $(\ell,k)$ a simplicial, such 
$I\times J$ is a maximal rectangle and the only maximal rectangle of $\bmat{Y}$ that covers the simplicial $1$ at $(\ell,k)$.
%A \emph{simplicial vertex} of a graph is one whose neighbours form a clique. Observe that if $(\ell,k)$ is a simplicial $1$ of $\bmat{Y}$, then it is a simplicial vertex of $\GG(\bmat{Y})$, but the converse does not necessarily hold as $(\ell,k)$ can still be a simplicial vertex with $y_{\ell j}=y_{i k}=?$ and $y_{ij}=0$.
To \emph{remove the simplicial $1$} at $(\ell,k)$ of $\bmat{Y}$ we delete row $\ell$ and column $k$ and set all remaining entries that are in $I\times J$ to $?$s. 
%\ooo{So if $\bmat{Y}'$ is obtained by removing the simplicial $1$ at $(\ell,k)$ then $\bmat{Y}'$ is the submatrix of $\bmat{Y}^{I\times J}$ obtained by dropping row $\ell$ and column $k$ which do not contain any $1$s.}
%In $\GG(\bmat{Y})$, the removal of simplicial $1$ at $(\ell,k)$ is equivalent to deleting all the vertices of the clique that is formed by $(\ell,k)$ and its neighbours. %\ooo{do I have to mention german guy?}

%\begin{equation}
%\begin{bmatrix}
% 1 & 1 & ? & 0 \\
% ? & ? & 0 & 1 \\
% 0 & 1 & ? & 1
%\end{bmatrix}
%\;
%\begin{bmatrix}
%1 & 1 & ? & 0 & 0\\
%? & 1 & 1 & 1 & 0\\
%? & ? & 0 & 0 & 1\\
%0 & ? & 1 & 0 & 0\\
%1 & ? & ? & 1 & ?
%\end{bmatrix}
%\end{equation}

\begin{lemma}
	\label{lemma_removing_simplicial_rectangle}
%	If $(\ell,k)$ is a simplicial $1$ of a generalised binary matrix $\bmat{Y}$ and $I\times J$ is the unique maximal rectangle  of $\bmat{Y}$ that covers $(\ell,k)$ then $i(\bmat{Y})=i(\bmat{Y}^{I\times J}) + 1$ and $br(\bmat{Y})=br(\bmat{Y}^{I\times J}) + 1$.
	If $\bmat{\emph{Y}}'$ is obtained by removing a simplicial $1$ of a generalised binary matrix $\bmat{\emph{Y}}$, then $i(\bmat{\emph{Y}})=i(\bmat{\emph{Y}}')+1$ and $br(\bmat{\emph{Y}})=br(\bmat{\emph{Y}}')+1$.
\end{lemma}
\begin{proof} 
	Let $(\ell,k)$ be the simplicial $1$ and $I\times J$ its unique maximal rectangle.
	For a maximum isolated set $S'$ and a minimum rectangle cover $\RR'$ of $\bmat{Y}'$,
	$S'\cup \{(\ell,k) \}$ and $\RR' \cup (I\times J)$ are clearly feasible for $\bmat{Y}$.
	Conversely, if $S$ is a maximum isolated set of $\bmat{Y}$, then $S\cap (I\times J) = \{(i,j) \}$ for some $(i,j)\in I\times J$, as otherwise $S\cup \{(\ell, k)\}$ would be a larger isolated set of $\bmat{Y}$. So $S\setminus\{(i,j) \}$ is a feasible isolated set of $\bmat{Y}'$. As $(\ell,k)$ is a simplicial $1$, we may assume that $I\times J$ is used in a minimum cover $\RR$ of $\bmat{Y}$. Then $\RR\setminus \{I\times J \}$ is a feasible cover of $\bmat{Y}'$.\hfill $\square$
\end{proof}

Our definition of simplicial $1$s for a standard binary matrix $\bmat{X}$ is identical to the definition of \emph{bisimplicial edges} \cite{Golumbic:2004} in the bipartite graph whose biadjacency matrix is $\bmat{X}$.
%setting. %, a bisimplicial edge is an edge whose endpoints' neighbours form a biclique.  
The key difference  is how we remove a simplicial $1$ and transition into generalised binary matrices.
%The \emph{elimination of a bisimplicial edge} in matrix terms corresponds to dropping the corresponding $1$'s row and column and leaving the rest of the matrix unchanged, and is used to show that a rectangle with a maximum number of $1$s is polynomial time solvable for totally balanced matrices \cite{Muller:1996:EdgePerfect,Golumbic:2004}.

We have seen that not every induced subgraph of $\GG(\bmat{X})$ corresponds to a submatrix of $\bmat{X}$, but by turning $1$s to $?$s we can consider arbitrary induced subgraphs of $\GG(\bmat{X})$ in matrix form.
The idea behind the next matrix operation is to expose induced subgraphs of rectangle cover graphs  without explicitly setting matrix entries to $?$s. Let $\bmat{X}\in \B^{m\times n}$. By \emph{stretching} a $1$ at $(\ell,k)\in\supp(\bmat{X})$ we get the ${(m+1)\times (n+1)}$ binary matrix $\SSS^{(\ell,k)} (\bmat{X})$   which satisfies
\begin{align}
\mathcal{S}^{(\ell,k)}(\bmat{X})_{i,j}&= x_{i,j} && i\in [m], j\in[n],\\
\mathcal{S}^{(\ell,k)}(\bmat{X})_{i,j}&= 1		& &(i,j)\in \{ (\ell, n+1), (m+1,k), (m+1,n+1)  \},
%\mathcal{S}^{(\ell,k)}(\bmat{X})_{i,j}&= 0		&& \text{otherwise}.
\end{align}
and $\mathcal{S}^{(\ell,k)}(\bmat{X})_{i,j}=0$ otherwise. 
For instance, if $(m,n)\in\supp(\bmat{X})$ then by stretching $(m,n)$ we obtain
\begin{equation}
\label{eq_stretching_def}
\mathcal{S}^{(m,n)}(\bmat{X})=\begin{bsmallmatrix}
x_{1,1} &  \dots&& x_{1,n}& 0\\
\vdots & \ddots&&\vdots& \vdots\\
&& &x_{m-1,n} & 0\\
x_{m,1}  &\dots &x_{m, n-1}& 1& 1\\
0 & \dots & 0 & 1 & 1
\vrule width 0pt depth 3.2pt
\end{bsmallmatrix}.
\end{equation}
Stretching $(\ell,k)$ adds in a simplicial $1$ at position $(m+1,n+1)$ whose unique maximal rectangle covers only $(\ell,k)$ from $\supp(\bmat{X})$. By Lemma \ref{lemma_removing_simplicial_rectangle}, removing the simplicial $1$ at $(m+1,n+1)$, %and observing that adding in or dropping rows and columns which have no $1$s does not change $i(\bmat{Y})$ and $br(\bmat{Y})$
 we get 
\begin{equation}
\label{eq_stretching_iso_br}
i(\SSS^{(\ell,k)} (\bmat{X})) = i(\bmat{X}^{(\ell,k)})+1,\quad  br(\SSS^{(\ell,k)} (\bmat{X})) = br(\bmat{X}^{(\ell,k)})+1,
\end{equation}
where $\bmat{X}^{(\ell,k)}$ is a shorter notation for $\bmat{X}^P$ with $P=\{(\ell,k) \}$.

For a non-empty set $Q\subseteq \supp(\bmat{X})$, the matrix obtained by stretching each $1$ in $Q$ is denoted by $\mathcal{S}^Q(\bmat{X})$. We adopt the convention to stretch $1$s in $Q$ in non-decreasing order of row and then column index, so $\mathcal{S}^Q(\bmat{X})$ may be written in block form as
\begin{equation}
\label{eq_stretching_convention}
\mathcal{S}^Q(\bmat{X}) = 
\begin{bmatrix}
\bmat{X} & \bmat{U}\\
\bmat{L} &\;\; \bmat{I}_{|Q|}
\end{bmatrix}
\end{equation}
where $\bmat{U}$ is an $m \times |Q|$ matrix with $|Q|$ $1$s exactly one in each column that have non-decreasing row index from left to right, $\bmat{L}$ is an $|Q|\times n$ matrix with $|Q|$ $1$s exactly one in each row and $\bmat{I}_{t}$ is the $t\times t$ identity matrix.
%If we wish to then stretch a $1$ of matrix $\mathcal{S}^Q(\bmat{X})$ that is created by stretching, we denote that by iterating the stretching operator. 

If $(\ell,k)$ is a simplicial $1$ of $\bmat{X}$ then we say that $\SSS^{(\ell,k)} (\bmat{X})$ is obtained by \emph{simplicial stretching}. Looking at $\GG(\bmat{X})$ and using  Lemma \ref{lemma_removing_simplicial_rectangle}  and the  Clique Cutset Lemma \cite{Berge:1989:Hypergraphs:CombinatoricsOfFiniteSets}, the following can be proved.
\begin{lemma} Let $\bmat{X}$ be superfirm. Then $\SSS^{(\ell,k)}(\bmat{X})$ is firm.
	% but not necessarily superfirm. 
	Furthermore, if $(\ell,k)$ is a simplicial $1$ of $\bmat{X}$, then $\SSS^{(\ell,k)}(\bmat{X})$ is superfirm.
%	If $\bmat{X}$ is superfirm then $\SSS^{(\ell,k)}(\bmat{X})$ is firm but not necessarily superfirm.
%	If $\SSS^{(\ell,k)}(\bmat{X})$ is obtained by simplicial stretching from superfirm $\bmat{X}$ then it is superfirm.
\end{lemma}

This lemma is tight in two ways. First, non-simplicial stretching may destroy superfirmness. Second, both simplicial and non-simplicial stretching do not preserve firmness.
In Section \ref{section_mnf}, we will exploit the superfirmness and firmness destroying properties of stretching to create minimally non-firm matrices.

For $n\ge 3$, let $\bmat{C}_n\in\B^{n\times n}$ be  the $n$-th \emph{cycle matrix} with exactly two $1$s in each row and  column such that no proper submatrix has this property. A binary matrix is \emph{totally balanced} if it has no $\bmat{C}_n$ submatrices for any $n\ge3$. Totally balanced matrices are exactly those that have a $\mvec{\Gamma}$-free ordering \cite{Lubiw:1987:Doubly}, where $\mvec{\Gamma}=\begin{bsmallmatrix}
1 & 1 \\
1 & 0
\end{bsmallmatrix}$. The following result can be verified by a $\mvec{\Gamma}$-free ordering.% of $\SSS^{Q}(\bmat{X})$.
\begin{lemma}
	\label{lemma_stretching_preserves_TB}
	If $\bmat{X}$ is totally balanced then so is $\SSS^Q(\bmat{X})$ for any $Q\subseteq\supp(\bmat{X})$. 
\end{lemma}
%$$\bmat{C}_n 
%= 
%\begin{bmatrix}
%\vrule width 0pt height 6pt 
%1  & 1 &      \\
%& 1 & 1   & &   \\
%& &  \scriptsize{\ddots}&\scriptsize{\ddots}  &\\
%& &  & 1&  1 \\
%1&&   &   & 1 
%\vrule width 0pt depth 3.2pt
%\end{bmatrix}$$

\section{Superfirm Matrices and Odd Holes}

\label{section_theorem_odd_antihole}

The Strong Perfect Graph Theorem \cite{Chudnovsky:2006:SPGT} tells us that a binary matrix $\bmat{X}$ is superfirm if and only if $\GG(\bmat{X})$ has no odd holes and no odd antiholes. 
But which are the necessary submatrices so that odd holes or odd antiholes appear in $\GG(\bmat{X})$? 
In this section, we show that forbidding odd antiholes in $\GG(\bmat{X})$ is unnecessary. 
Then we study when a $5$-hole in $\GG(\bmat{X})$ exists.

A theorem of Lubiw in \cite{Lubiw:1990} states that for $\GG(\bmat{X})$ to have an odd antihole of size $7$ or more, $\bmat{X}$ needs to have the $3\times 3$ cycle matrix $\bmat{C}_3$ as a submatrix. Note that $\bmat{C}_3$ is superfirm.
Let $\bmat{1}$ be the all $1$s column vector of appropriate size 
and define 
$\bmat{W}:=\begin{bsmallmatrix}
\bmat{C}_3 & \mvec{1} \\
\mvec{1}^\top & 1
\end{bsmallmatrix}$ and
 $\overline{\bmat{I}}_4:=\begin{bsmallmatrix}
\bmat{C}_3 & \mvec{1} \\
\mvec{1}^\top & 0
\end{bsmallmatrix}$ in $\B^{4\times4}$.
Considering a slight extension of Lubiw's proof, we show that these two larger matrices are necessary for the appearance of odd antiholes.
\begin{lemma}
	\label{lemma_no_C3_no_oddhole}
	If $\GG(\bmat{X})$ contains an odd antihole of size $7$ or more then
	$\bmat{X}$ has $\bmat{W}$ or $\overline{\bmat{\emph{I}}}_4$ as a submatrix.
\end{lemma}
\begin{proof} Following the proof structure of \cite[Theorem 6.3]{Lubiw:1990},
suppose that $\bmat{X}$ has no such submatrices but $\GG(\bmat{X})$ contains an antihole $A\subseteq \supp(\bmat{X})$ of odd size $k=|A|\ge 7$. By duplicating rows and columns of $\bmat{X}$, we may assume that no two $1$s in $A$ are in the same row or column. Note that row and column duplication cannot introduce $\bmat{W}$ or $\overline{\bmat{I}}_4$ submatrices into $\bmat{X}$.
Then the submatrix $\bmat{X}'$ of $\bmat{X}$ that consists of the rows and columns of the $1$s in $A$ is of dimension $k\times k$ and may be permuted so that the vertices of $A$ appear on the main diagonal and are non-adjacent to the two vertices that are directly above and below them. Then $\bmat{X}'$ has the form as below where each undecided entry pair $(i,j),(j,i)$ denoted by $*$s satisfies $|\supp(\bmat{X}') \cap \{(i,j),(j,i)  \}|\le 1$ so that $A$ is indeed an antihole in $\GG(\bmat{X})$.
\begin{equation}
\label{eq_odd_anithole_C_3_submatrix_proof}
\bmat{X}'=
\begin{bsmallmatrix}
		\vrule width 0pt height 6pt 
1 & * & 1 &1  & \dots &1& 1 & *\\
* & 1 & * &1  &  &  1&1 & 1\\
1 & * & 1&*   &  &  1&1 & 1\\
1&1 & * & 1   &  &  1&1 & 1\\
\vdots &    &  &&\ddots   &&&\vdots \\
1 & 1& 1   &1&  &1&  * & 1\\
1 &1 & 1   & 1& &*& 1 & *\\
* & 1& 1  &1& \dots&1 & * & 1
		\vrule width 0pt depth 3.2pt
\end{bsmallmatrix}
\quad 
\Rightarrow
\quad 
\begin{bsmallmatrix}
\vrule width 0pt height 6pt 
1 & 0 & 1 &1  & \dots &1& 1 & *\\
1 & 1 & 1 &1  &  &  1&1 & 1\\
1 & 0 & 1&0   &  &  1&1 & 1\\
1&1 & 1 & 1   &  &  1&1 & 1\\
\vdots &    &  &&\ddots   &&&\vdots \\
1 & 1& 1   &1&  &1&  * & 1\\
1 &1 & 1   & 1& &*& 1 & *\\
* & 1& 1  &1& \dots&1 & * & 1
\vrule width 0pt depth 3.2pt
\end{bsmallmatrix}
%\begin{bsmallmatrix}
%1 & 0& 1   & \dots &1& 1 & *\\
%1 & 1 & 1   &  &  1&1 & 1\\
%1 & 0 & 1   &  &  1&1 & 1\\
%\scriptsize{\vdots}  &    &  &\scriptsize{\ddots}    &&&\scriptsize{\vdots} \\
%1 & 1& 1   &  &1&  0 & 1\\
%1 &1 & 1   &  &1& 1 & 1\\
%* & 1& 1  & \dots&1 & 0 & 1
%\end{bsmallmatrix}
\end{equation}
Assume without loss of generality that $x'_{1,2}=0$. 
Suppose that $x'_{2,3}=0$.
If $x'_{5,6}=x'_{6,5}=0$ then the submatrix $I\times J$ of $\bmat{X}'$ with $I=\{1,2,5,6  \}$, $J=\{2,3,5,6 \}$ is $\overline{\bmat{I}}_4$.
Moreover, if $x'_{5,6}+x'_{6,5}=1$ then $I\times J$ is $\bmat{W}$. Hence,  $x'_{2,3}\not=0$. In general, exactly one of $(i,j)$ and $(i+1,j+1)$ can be a $0$ for all $*$s, so the zeros of $\bmat{X}'$ must zigzag as shown in the right of Equation \eqref{eq_odd_anithole_C_3_submatrix_proof}. 
But as $k$ is odd, this is impossible.$\hfill \square$
\end{proof}

The importance of Lemma \ref{lemma_no_C3_no_oddhole} over Lubiw's theorem, is that both $\bmat{W}$ and $\overline{\bmat{I}}_4$ contain the submatrix $\bmat{H}_3:=[\mvec{1},\bmat{C}_3]$ and $\GG(\bmat{H}_3)$ contains three $5$-holes as shown in Figure \ref{figure_H3}, whereas $\bmat{C}_3$ is superfirm. This shows that a rectangle cover graph cannot contain an odd antihole of size $7$ or larger if it does not contain an odd hole. Recalling that a $5$-antihole is just a $5$-hole, we obtain the following result.
\begin{figure}[tb!]
	     \centering
	\begin{subfigure}[b]{0.3\textwidth}
		\centering
		\includegraphics[scale=0.65]{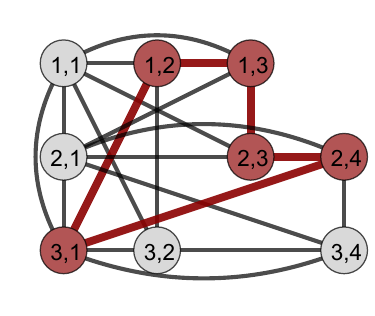}
	\end{subfigure}
%	\hfill
	\begin{subfigure}[b]{0.3\textwidth}
		\centering
		\includegraphics[scale=0.65]{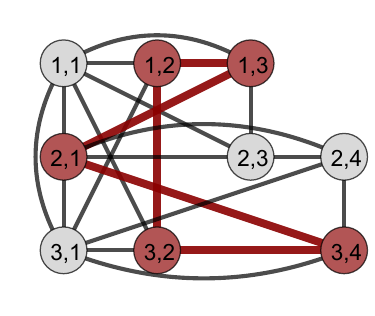}
	\end{subfigure}
%	\hfill
	\begin{subfigure}[b]{0.3\textwidth}
		\centering
		\includegraphics[scale=0.65]{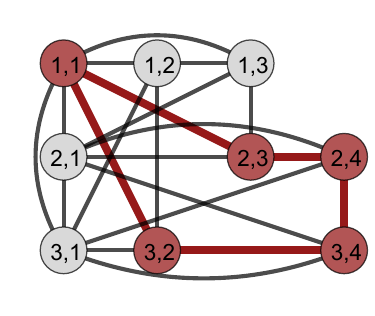}
\end{subfigure}	
\caption{\label{figure_H3} The three $5$-holes in $\GG(\bmat{H}_3)$}
\end{figure}
\begin{theorem}
	\label{theorem_no_odd_holes_then_superfirm}
	$\bmat{X}$ is superfirm if and only if $\GG(\bmat{X})$ has no odd holes.
\end{theorem}
%Therefore, one needs only focus on odd holes to give a forbidden submatrix characterisation of superfirm matrices. 

Theorem \ref{theorem_no_odd_holes_then_superfirm} motivates us to study when $\GG(\bmat{X})$ has odd holes.
We initialise this by characterising when a $5$-hole exists in $\GG(\bmat{X})$.
The proof is skipped %due to space constraints 
but it is of similar nature to that of Lemma \ref{lemma_no_C3_no_oddhole}.
Let $\bmat{K}_5\in\B^{5\times 5}$ be the circulant matrix with exactly three $1$s per row and column and recall $\bmat{D}_4$ from Figure \ref{figure_D4}.
\begin{theorem}
	\label{theorem_5hole_iff_submatrix}
	$\GG(\bmat{X})$ contains a $5$-hole if and only if $\bmat{X}$ has at least one of $\bmat{D}_4$, $\bmat{H}_3$, $\bmat{H}_3^\top$ or  $\bmat{K}_5$ as a submatrix.%, where $\bmat{K}_5\in\B^{5\times 5}$ is the circulant matrix with exactly three $1$s per row and column.
\end{theorem}

\section{Four Infinite Classes of Minimally Non-Firm Matrices}
\label{section_mnf}

In this section we prove a theorem which shows how minimally non-firm matrices may arise by using the stretching operation. Then using  this theorem we show that four infinite classes of matrices are minimally non-firm.

Recall that a standard binary matrix $\bmat{X}$ is \emph{minimally non-firm (mnf)} if it is not firm but all proper submatrices of it are.
This definition naturally extends to generalised binary matrices $\bmat{Y}$, $\bmat{Y}$ is mnf if $i(\bmat{Y})<br(\bmat{Y})$ and $i(\bmat{Y}')=br(\bmat{Y}')$ for all proper submatrices $\bmat{Y}'$ of $\bmat{Y}$. Note that as $br(\bmat{Y})$ and $i(\bmat{Y})$ are invariant under transposition, the transpose of any mnf matrix is mnf as well.
The following two simple results apply to both standard and generalised mnf matrices.
\begin{lemma}
	\label{lemma_mnf_at_least_two_1s}
	Each row and column of an mnf matrix has at least two non-zeros.
\end{lemma}
\begin{proof}
	Suppose $\bmat{Y}$ is mnf and its $i$-th row only has a single nonzero at entry $(i,j)$. If $y_{i,j}=?$ then row $i$ can clearly be dropped without changing $i(\bmat{Y})$ or $br(\bmat{Y})$. 
	If $y_{i,j}=1$ then $(i,j)$ is a simplicial $1$. By Lemma \ref{lemma_removing_simplicial_rectangle}, removing it we obtain a firm submatrix $\bmat{Y}'$ with $i(\bmat{Y})-1=i(\bmat{Y}')=br(\bmat{Y}')=br(\bmat{Y}) -1$, a contradiction.\hfill $\square$
\end{proof}
\begin{lemma}
	\label{lemma_mnf_diff_1}
	If $\bmat{Y}$ is mnf  then $i(\bmat{Y})=br(\bmat{Y})-1$.
\end{lemma}
\begin{proof} Let $\bmat{Y}$ be mnf. By deleting a single row or column of $\bmat{Y}$ we get a submatrix $\bmat{Y}'$ which by definition is firm and  satisfies $i(\bmat{Y})-1 \le i(\bmat{Y}') \le i(\bmat{Y})$ and $br(\bmat{Y})-1 \le br(\bmat{Y}')\le br(\bmat{Y})$
%	\begin{equation}
%	\label{eq_delete_row_col_iso_rank}
%	 \qquad \text{ and } \qquad ,
%	\end{equation}
	as a row or column forms a rectangle and may contain at most one element of an isolated set.
%	$i(\bmat{X})-1\le i(\bmat{X}') \le i(\bmat{X})$ and $br(\bmat{X})-1\le br(\bmat{X}') \le br(\bmat{X})$
	 So, we must have $br(\bmat{Y}')=br(\bmat{Y})-1$ as otherwise $\bmat{Y}$ is firm. 
	 But then $i(\bmat{Y}')=br(\bmat{Y}')=br(\bmat{Y})-1 \le i(\bmat{Y})$
%	\begin{equation*}
%	i(\bmat{X})-1 \le i(\bmat{X}') = br(\bmat{X}')=br(\bmat{X})-1 \le i(\bmat{X}).%  \qquad \text{ and }.
%	\end{equation*}
	which together with $i(\bmat{Y})<br(\bmat{Y})$ implies $br(\bmat{Y})-1 = i(\bmat{Y})$. \hfill $\square$
\end{proof}

%\ooo{Should these be stated? 
%\begin{conjecture}
%	If $\bmat{X}\in \B^{m\times n}$ is minimally non-firm, then $|n-m|\le 1$.
%\end{conjecture}
%\begin{conjecture}
%	If $\bmat{X}\in \B^{m\times n}$ is minimally non-superfirm, then it's firm.
%\end{conjecture}
%}
%What is the smallest dimension of a mnf matrix $\bmat{X}$ can have? 

By Theorem \ref{theorem_no_odd_holes_then_superfirm}, $\bmat{X}$ is superfirm if $\GG(\bmat{X})$ has no odd holes, so for $\bmat{X}$ to be mnf $\GG(\bmat{X})$ must contain odd holes. 
Using Theorem \ref{theorem_5hole_iff_submatrix} one can show that the smallest mnf standard binary matrices are of dimension $4\times 4$ and there are exactly two of them: $\overline{\bmat{I}}_4$ and $\overline{\bmat{I}}_4'$, where $\overline{\bmat{I}}_4'$ is obtained from $\overline{\bmat{I}}_4$ by turning a single $1$ to a $0$ (for instance at $(1,4)$, but due to symmetry any other $1$ would work). 

Let $\bmat{X}$ be a standard binary matrix with an odd hole $C$ in $\GG(\bmat{X})$ of size $|C|=2k+1$. Stretching all $1$s at $Q=\supp(\bmat{X})\setminus C$ of $\bmat{X}$, by Lemma \ref{lemma_removing_simplicial_rectangle} we get
\begin{equation}
i(\SSS^{Q}(\bmat{X})) -|Q|= i(\bmat{X}^Q) = k <k+1=br(\bmat{X}^Q)= br(\SSS^{Q}(\bmat{X})) -|Q|, 
\end{equation}
so $\SSS^Q(\bmat{X})$ is  non-firm. This recipe however, does not guarantee that $\SSS^Q(\bmat{X})$ is \emph{minimally} non-firm.
By adding extra conditions on $Q$, minimality can be enforced.
% and the next result shows how the stretching operation is used to create mnf standard binary matrices.
\begin{theorem}
	\label{theorem_mnnf_from_mnnf_generalised} Let $\bmat{X}\in\B^{m\times n}$.
	If $\bmat{X}^Q$ is a minimally non-firm generalised binary matrix for some non-empty $Q\subset \supp(\bmat{X})$ and $\bmat{X}^P$ is firm for all $P\subsetneq Q$, 
	then 
	$\mathcal{S}^Q(\bmat{X})\in \B^{(m+|Q|)\times(n+|Q|)}$ 
%	the $(m+|Q|)\times(n+|Q|)$ binary matrix $\mathcal{S}^Q(\bmat{X})$ 
	is minimally non-firm.
	%\in\B^{}$ obtained from $\bmat{X}$ by stretching $Q$  is mnf.
\end{theorem}
	\begin{proof}
	$\mathcal{S}^Q(\bmat{X})$ may be written as a block matrix with four blocks $\bmat{X}, \bmat{L},\bmat{U}$ and $\bmat{I}_{|Q|}$ as in Equation \eqref{eq_stretching_convention}.
	By construction all $1$s in block $\bmat{I}_{|Q|}$ are simplicial, hence removing them 
	we obtain the mnf generalised binary matrix  $\bmat{X}^{Q}$.
	By Lemma \ref{lemma_removing_simplicial_rectangle} then  $i(\mathcal{S}^Q(\bmat{X}))=i(\bmat{X}^Q)+|Q|< br(\bmat{X}^Q)+|Q|=br(\mathcal{S}^Q(\bmat{X}))$.
	%, which implies $i(\mathcal{S}^Q(\bmat{X})) < br(\mathcal{S}^Q(\bmat{X}))$.

	Suppose that not all proper submatrices of $\mathcal{S}^Q(\bmat{X})$ are firm and let $\bmat{Y}$ be the smallest non-firm proper submatrix indexed by $I\times J$.
	%Let  $I\subset [m+|Q|]$ and $J\subset [n+|Q|]$ contain the row and column indices of $\bmat{Y}$ respectively. 
	%\ooo{I may need to change notation here so that $I$ is not confused with $\bmat{I}_{|Q|}$.} 
	Then $\bmat{Y}$ is mnf.
	Note that the four block matrices of $\mathcal{S}^Q(\bmat{X})$ are all firm: (1) $\bmat{X}$ is firm as it is just $\bmat{X}^\emptyset$. (2) $\bmat{I}_{|Q|}$ is clearly firm. (3) $\bmat{U}$ has exactly one $1$ per column, so it can be obtained from an identity matrix by duplicating columns and adding zero rows, and thus firm. (4) Similarly, as $\bmat{L}$ has exactly one $1$ per row, it is firm.
	Hence $\bmat{Y}$ cannot be fully contained in any of the four blocks.
	As $\bmat{Y}$ is a mnf standard binary matrix it has at least two $1$s in each row and column by Lemma \ref{lemma_mnf_at_least_two_1s}.
	Since block $\left[ \begin{smallmatrix}
	\bmat{L} & \bmat{I}_{|Q|}
	\end{smallmatrix}\right]$ 
	has exactly two $1$s in each row, 
	if $\bmat{Y}$ has a row from this block, then $\bmat{Y}$ must also contain the columns of both $1$s in this row.  
	Similarly, if $\bmat{Y}$ contains a column from block $\left[ \begin{smallmatrix}
	\bmat{U} \\ \bmat{I}_{|Q|}
	\end{smallmatrix}\right]$, it must contain the rows of both $1$s in this column.
	Therefore, the rows in $I$ from block $\left[ \begin{smallmatrix}
	\bmat{L} & \bmat{I}_{|Q|}
	\end{smallmatrix}\right]$ and the columns in $J$ from block $\left[ \begin{smallmatrix}
	\bmat{U} \\ \bmat{I}_{|Q|}
	\end{smallmatrix}\right]$ come in pairs and may be identified with their $1$ in block $\bmat{I}_{|Q|}$.
	Let $P$ be the subset of $Q$ whose stretching created the $1$s in block $\bmat{I}_{|Q|}$ which are in $\bmat{Y}$.
	Removing all $|P|$ simplicial $1$s present in $\bmat{Y}$ from block $\bmat{I}_{|Q|}$ we obtain a generalised binary matrix which is fully contained in block $\bmat{X}$ and is just a submatrix $\bmat{Z}$ of $\bmat{X}^{P}$. 
	By Lemma \ref{lemma_removing_simplicial_rectangle}, $\bmat{Z}$ satisfies $i(\bmat{Z}) + |P| = i(\bmat{Y})$ and $br(\bmat{Z}) + |P|=br(\bmat{Y})$.
	If $P=Q$, then $I$ contains all the rows and columns from block $\bmat{I}_{|Q|}$ so $\bmat{Z}$ must be a proper submatrix of $\bmat{X}^Q$, hence firm.
	If $P\not = Q$, then $\bmat{Z}$ is a submatrix of the firm matrix $\bmat{X}^P$.
	In both cases $i(\bmat{Z})=br(\bmat{Z})$ which implies $i(\bmat{Y})=br(\bmat{Y})$, a contradiction.\hfill $\square$
\end{proof}

One can see that a partial converse of the above theorem also holds, i.e. if a standard binary mnf matrix has some simplicial $1$s then by removing those we obtain a generalised binary mnf  matrix for which the theorem's conditions hold. Note however, that not all mnf matrices have simplicial $1$s, e.g. $\overline{\bmat{I}}_4$, hence certainly not all mnf matrices arise via Theorem \ref{theorem_mnnf_from_mnnf_generalised}.
%\ooo{
%\begin{conjecture}
%	If a minimally non-firm matrix has some simplicial $1$s, then it is obtained by stretching.
%\end{conjecture}}

%Next with the motivation to use Theorem \ref{theorem_mnnf_from_mnnf_generalised} we prove that $3$ classes of matrices satisfy the conditions of the theorem.
Recall  $\bmat{C}_n$ is the $n\times n$ cycle matrix.
For $n\ge 3$, let $\bmat{M}_{n+1}:=\SSS^{(n,n)}(\bmat{C}_n)$ be the $(n+1)\times (n+1)$ matrix and $\bmat{H}_n:=[\bmat{1}, \bmat{C}_n]$ be the $n\times (n+1)$ matrix,
% as below,
%$\bmat{M}_n, \bmat{H}_n$ and $\bmat{D}_n$ are shown below.
 \begin{align}
 \label{eq_matrices}
 \bmat{M}_n %= \SSS^{(n-1,n-1)} (\bmat{C}_{n-1})
 &= 
 \begin{bsmallmatrix}
 \vrule width 0pt height 6pt 
 1  & 1 &      \\
 & 1 & 1   & &   \\
 & &  \ddots&\ddots  &\\
 & &  & 1&  1 \\
 1&&   &   & 1 &1\\ 
 &&   &   & 1 &1
 \vrule width 0pt depth 3.2pt
 \end{bsmallmatrix}, %\in\B^{(n+1)\times(n+1)},
&
 \bmat{H}_n %= \begin{bmatrix}\mvec{1} & \bmat{C}_n\end{bmatrix}
& =
  \begin{bsmallmatrix}
 \vrule width 0pt height 6pt 
 1 &1  & 1 &      \\
1 & & 1 & 1   & &   \\
 \vdots& & &  \ddots&\ddots  &\\
 1&& &  & 1&  1 \\
 1&1&&   &   & 1 
 \vrule width 0pt depth 3.2pt
 \end{bsmallmatrix}.
 \end{align}
 Matrices $\bmat{M}_n$ appear in the work of Lubiw \cite{Lubiw:1990} as forbidden submatrices for a subset of superfirm matrices that can be decomposed into linear matrices by applying the matrix equivalent of split decomposition \cite{Cunningham:1980} on bipartite graphs. 
 
 Recall matrix $\bmat{D}_4$ from Figure \ref{figure_D4} and for $n\ge 5$, let $\bmat{D}_n := \mathcal{S}^{(3,n-1)}(\bmat{D}_{n-1})$. In addition, let $\bmat{T}_5\in\B^{5\times 5}$ as below and for $n\ge 6$ define $\bmat{T}_n:=\SSS^{(4,n-1)} (\bmat{T}_{n-1})$,
 \begin{align}
 \bmat{D}_n=
 \begin{bsmallmatrix}
  \vrule width 0pt height 6pt 
 & 1 & 1 &   \\
 1 & 1 & 1  &   \\
 1& 1 & 1 & 1 &\dots & 1 \\
&&1  & 1 &      \\
&&&   \ddots&\ddots  \\
&&&  & 1&  1 
 \vrule width 0pt depth 3.2pt
 \end{bsmallmatrix},
&&
   \bmat{T}_5=
 \begin{bsmallmatrix}
   \vrule width 0pt height 6pt 
 &1&  & 1 &   \\
 1& & 1 &   &   \\
 & 1& 1 & 1  &   \\
 1&& 1 & 1 & 1\\
 &&&1&1
  \vrule width 0pt depth 3.2pt
 \end{bsmallmatrix},
&&
  \bmat{T}_n=
 \begin{bsmallmatrix}
    \vrule width 0pt height 4pt 
 &1&  & 1 &   \\
 1& & 1 &   &   \\
  & 1& 1 & 1  &   \\
 1&& 1 & 1 & 1 &\dots & 1 \\
 &&&1  & 1 &      \\
 &&&&   \ddots&\ddots  &\\
 &&&&  & 1&  1 
  \vrule width 0pt depth 2.pt
 \end{bsmallmatrix}.
 \end{align}
 All these matrices contain odd holes in their rectangle cover graph as shown in Figure \ref{figure_H3} for $\bmat{H}_3$ and Figure \ref{figure_mixed} for $\bmat{M}_4,\bmat{D}_5$ and $\bmat{T}_6$.
% $\GG(\bmat{M}_n)$ contains a single $2n-1$-hole as shown in Figure \ref{figure_M4} for $n=4$, and 
%% As already shown in Figure \ref{figure_H3}, $\GG(\bmat{H}_3)$ contains three $5$-holes. In general, 
% $\GG(\bmat{H}_n)$ contains $n$ $2n-1$-holes, one of them is shown in Figure \ref{figure_H4} for $n=4$ and for $n=3$ all three of them are shown in Figure \ref{figure_H3}.
% $\GG(\bmat{D}_4)$ is shown in Figure 
% $(n,n)$ for 
In the remaining parts of this section, we will prove that by choosing the set $Q$ in Theorem \ref{theorem_mnnf_from_mnnf_generalised} to be $\{(n,n)\}$ for $\bmat{M}_n$, $G_n=\{(n,2),(n,n+1)  \}$ for $\bmat{H}_n$, $Q_n=\{ (1,2),(2,1),(n,n) \}$ for $\bmat{D}_n$ and $\bmat{T}_n$, the conditions of Theorem \ref{theorem_mnnf_from_mnnf_generalised} are satisfied and thus we get our main theorem.
 \begin{theorem}
 	\label{theorem_four_classes_of_mnf_matrices}
 	For $n\ge 4$, $\SSS^{(n,n)}(\bmat{M}_n)$, $\SSS^{G_{n-1}}(\bmat{H}_{n-1})$, $\SSS^{Q_n}({\bmat{D}}_n)$ and $\SSS^{Q_{n+1}}({\bmat{T}}_{n+1})$ are mnf standard binary matrices. In addition, $\SSS^{Q_n}({\bmat{D}}_n)$ and $\SSS^{Q_{n+1}}({\bmat{T}}_{n+1})$ are totally balanced.
 \end{theorem}

The claim of total balancedness is immediate by Lemma \ref{lemma_stretching_preserves_TB} as $\bmat{D}_4$ is an interval matrix and $\bmat{T}_5$ is $\mvec{\Gamma}$-free \cite{Lubiw:1987:Doubly}.
Lubiw observed that  $\SSS^{Q_4}(\bmat{D}_4)$ and $\SSS^{Q_5}(\bmat{D}_5)$ are non-firm \cite{Lubiw:1990}. %in fact $\SSS^{(Q_4)}(\bmat{D}_4)$ is the 'swath' matrix of a rectilinear polygon (see .
Her observation served as a motivation to us to define the stretching operation and matrices $\bmat{D}_n$. 
\begin{figure}[tb!]
	\centering
	\begin{subfigure}[b]{0.25\textwidth}
		\centering
		\includegraphics[scale=0.7]{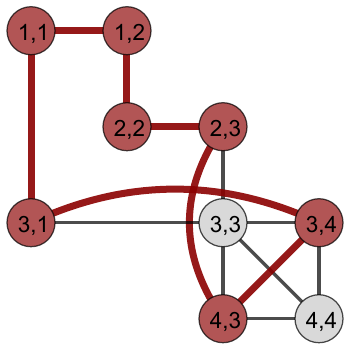}
		\caption{\label{figure_M4}$\GG(\bmat{M}_4)$}
	\end{subfigure}
	%	\hfill
%	\begin{subfigure}[b]{0.3\textwidth}
%		\centering
%		\includegraphics[scale=0.9]{figures/H4_hole.pdf}
%		\caption{\label{figure_H4}$\GG(\bmat{H}_4)$}
%	\end{subfigure}
		\hfill
	\begin{subfigure}[b]{0.3\textwidth}
		\centering
		\includegraphics[scale=0.75]{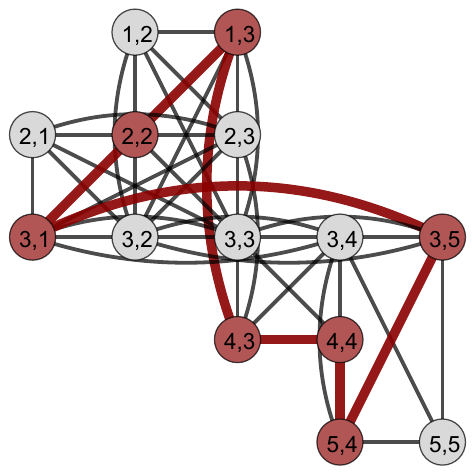}
		\caption{\label{figure_D5}$\GG(\bmat{D}_5)$}
	\end{subfigure}	
		\hfill
	\begin{subfigure}[b]{0.35\textwidth}
	\centering
	\includegraphics[scale=0.78]{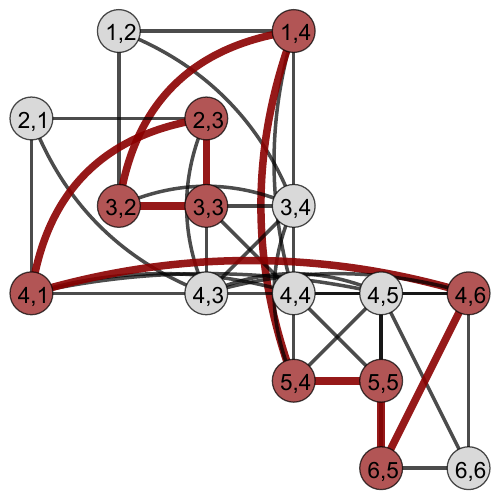}
	\caption{\label{figure_T6}$\GG(\bmat{T}_6)$}
\end{subfigure}
	\caption{\label{figure_mixed} Odd holes highlighted  in the rectangle cover graphs of $\bmat{M}_4$, $\bmat{D}_5$ and $\bmat{T}_6$}
\end{figure}
 
For the first two classes, $\bmat{M}_n$ and $\bmat{H}_n$, the proofs that Theorem \ref{theorem_mnnf_from_mnnf_generalised}'s conditions hold are similar because both  are minimally non-superfirm. A standard binary matrix is \emph{minimally non-superfirm (mnsf)} if  it is not superfirm but all proper submatrices of it are. Next, we show that the conditions hold for the class $\bmat{H}_n$.
%By the above Lemma all conditions of Theorem \ref{theorem_mnnf_from_mnnf_generalised} are satisfied, so the following holds.
%\begin{theorem} For $n\ge 4$,
%	$\mathcal{S}^{(n,n)}(\bmat{M}_n)\in \B^{(n+1)\times (n+1)}$ is minimally non-firm.
%\end{theorem}

%\subsection{Class II.}

%\begin{figure}[htb!]
%	\centering
%	\begin{subfigure}[b]{0.45\textwidth}
%		\centering
%		\includegraphics[scale=0.9]{figures/H4_hole.pdf}
%	\end{subfigure}
%	%	\hfill
%	\begin{subfigure}[b]{0.45\textwidth}
%		\centering
%		\includegraphics[scale=0.7]{figures/H5_hole.pdf}
%	\end{subfigure}	
%	\caption{\label{figure_H45} $\GG(\bmat{H}_n)$ with $\CC_n$ highlighted for $n=4,5$ }
%\end{figure}
%
	\begin{lemma}
	\label{lemma_1Cn_firmness_results}
	For $n\ge 3$,  $\bmat{H}_n^P$ is firm  for all $P\subsetneq G_n=\{(n,2),(n,n+1) \}$ and $\bmat{H}_n^{G_n}$ is a mnf generalised binary matrix. In addition, $\bmat{H}_n$ is mnsf.
	\end{lemma}
	\begin{proof}
		For $P\subseteq G_n$, at least $n$ rectangles are needed to cover $\bmat{H}_n^P$ as
		\begin{equation}
		\CC_n := \supp(\bmat{H}_n) 
		\setminus 
		(
		\{(i,1): i\in [n-1]  \}
%		\{(1,1),\dots, (n-1,1)  \}  
		\cup 
		G_n
%		\{(n,2),(n,n+1)  \}
		),
		\end{equation}
		is a $2n-1$-hole in $\GG(\bmat{H}_n^P)$.
		As $\bmat{H}_n^P$ only has $n$ rows, $br(\bmat{H}_n^P)=n$.
		
%		Observe that in $\GG(\bmat{H}_n)$, $(n,n+1)$ is adjacent to two consecutive vertices of $\CC_n$, and not adjacent to the other $2n-3$ vertices. The same holds for $(n,2)$.
		
		Note that submatrix $[1,n-1]\times [2,n+1]$ (where $[\ell,k]:=\{\ell,\ell+1,\dots,k\}$) has two isolated sets of size $n-1$.
		For $P\subsetneq G_n$, $(i,j)\in G_n\setminus P$ may be added to one of these two isolated sets to get an isolated set of size $n$ for $\bmat{H}_n^P$.
%
%		contains the submatrix $\bmat{C}_n^P$ whose rectangle cover graph contains a $2n-1$ path, so $i(\bmat{H}_n^P)=n$.
		%
		For $\bmat{H}_n^{G_n}$ however, none of the $1$s can be added to these two isolated sets, so we only have $i(\bmat{H}_n^{G_n})\ge n-1$.
%		the rectangle cover graph of submatrix $\bmat{C}_n^{G_n}$ is just a $2n-2$ path and gives an isolated set of size $n-1$.
		Suppose $\bmat{H}_n^{G_n}$ has an isolated set $T_n$ of size $n$. Then $T_n$ needs to contain a $1$ from each row, so $(n,1)\in T_n$. But then $T_n$ cannot contain $(1,2)$ and $(n-1,n+1)$, the only $1$s in columns $2$ and $n+1$, as they are in a rectangle with $(n,1)$. Hence $T_n$ has $n$ elements from $n-1$ distinct columns, which is impossible.
		
		$\GG(\bmat{H}_n)$ has no odd antiholes of size $7$ or more by Lemma \ref{lemma_no_C3_no_oddhole} but it contains $n$ $2n-1$-holes, one of which is $\CC_n$. Any other hole in $\GG(\bmat{H}_n)$ is either contained in the submatrix $\bmat{C}_n$ and hence it is the $2n$-hole, or contains at most two vertices from column $1$.  Note that if $(\ell,1)$ is a vertex of a hole then the hole cannot have another vertex from row $\ell$.
		If a hole contains a single vertex from column $1$ then  it is easy to see that it must be one of the $n$ $2n-1$-holes. If the hole has two vertices from column $1$, then it must contain an even number of vertices from submatrix $\bmat{C}_n$, so it is an even hole. Therefore, the $n$ $2n-1$-holes are the only odd holes in $\GG(\bmat{H}_n)$ which all have a vertex from every row and column.
		For $P\subseteq G_n$, $\GG(\bmat{Y})$ for any proper submatrix $\bmat{Y}$ of $\bmat{H}_n^P$ then has no odd holes and no odd antiholes, so $\GG(\bmat{Y})$ is perfect by the Strong Perfect Graph Theorem \cite{Chudnovsky:2006:SPGT}.$\hfill\square$
		%
%		Let $P\subseteq Q_n$, and suppose that $\bmat{H}_n^P$ has a proper submatrix which is not firm. Let $\bmat{Y}$ be the smallest such submatrix indexed by $I\times J$. Then $\bmat{Y}$ is mnf and must contain at least two non-zeros in each row and column by \ooo{Lemma \ref{lemma_mnf_at_least_two_1s}}. So if $2\in J$, then $1,n\in I$ and if $j\in J$ for $j\ge 3$ then $j-2,j-1\in I$.
	\end{proof}

Observe that $\GG(\bmat{M}_n)$ contains a single odd-hole of size $2n-1$ as shown in Figure \ref{figure_M4} for $m=4$. To prove that conditions of Theorem \ref{theorem_mnnf_from_mnnf_generalised} are satisfied by class $\bmat{M}_n$, the same structure of proof as for $\bmat{H}_n$ may be applied to get the following.
\begin{lemma}
	\label{lemma_stretchedCn_firmness_results}
	For $n\ge 4$, 
	$\bmat{M}_n$ is firm and mnsf, and $\bmat{M}_n^{(n,n)}$ is a mnf generalised binary matrix.
\end{lemma}

 Although $\bmat{D}_4$ and $\bmat{T}_5$ are mnsf, for larger $n$ as both $\bmat{D}_n$ and $\bmat{T}_n$ are defined recursively, they have proper submatrices which are not superfirm. 
 Hence the argument used in the proof of the previous two classes does not work for $\bmat{D}_n$ and $\bmat{T}_n$.
 Next we prove that class $\bmat{D}_n$ satisfies the conditions of Theorem \ref{theorem_mnnf_from_mnnf_generalised}.
\begin{lemma}
	\label{lemma_DnP_firmness_results}
	For $n\ge 4$, $\bmat{D}_n^P$ is firm for all $P\subsetneq Q_n=\{(1,2),(2,1),(n,n)  \}$ and $\bmat{D}_n^{Q_n}$ is a mnf generalised binary matrix. In addition, $\bmat{D}_4$ is mnsf. 
\end{lemma}
\begin{proof}
	
		\textbf{I.} 
	For all $P\subseteq Q_n$, $\GG(\bmat{D}_n^P)$ contains the $2n-3$-hole
	\begin{equation}
	\CC_n = \{(3,1), (2,2),(1,3), (4,3), \dots, (3,n)  \},
	\end{equation}
	 and thus $br(\bmat{D}_n^P)\ge n-1$. On the other hand, $\bmat{D}_n$ has a feasible cover using $n-1$ rectangles in which each row $i\not=3$ is covered by a distinct rectangle. %Hence $br(\bmat{D}_n^P)=n-1$.
	
	\textbf{II.} 
	In $\GG(\bmat{D}_n)$, each $(i,j)\in Q_n$ is adjacent to two consecutive vertices of $\CC_n$, and not adjacent to the others.
	For $P\subsetneq Q_n$, let $(\ell,k)\in Q_n\setminus P$ and $S_n$ be an  independent set of $\CC_n$ of size $n-2$ which does not use the two vertices of $\CC_n$ that are adjacent to $(\ell,k)$. Then  $S_n\cup \{(\ell,k) \}$ is a feasible isolated set of $\bmat{D}_n^P$.% of size $n-1$.

%	\textbf{Part III. }
	For  $\bmat{D}_n^{Q_n}$, $S_n$ is a feasible isolated set. %, but $\supp(\bmat{D}_n^{Q_n})\cap Q_n=\emptyset$.
%	 all three entries in $Q_n$ are set to $?$ so they cannot be used in an isolated set. 
%	Taking a maximum independent set in $\CC^{(n)}$ gives an isolated set of size $n-2$. 
	%
	Suppose that $\bmat{D}^{Q_n}_n$ has an isolated set $T_n$ of size $n-1$.
	Then as $\bmat{D}^{Q_n}_n$ is of size $n\times n$, there is exactly one row and one column that does not have a $1$ in $T_n$.
	Since columns $1$ and $n$ each have a single $1$ which are both in row $3$, exactly one of these $1$s must be in $T_n$.
		(a)
		Suppose that $(3,1)\in T_n$. 
		Then $(3,j)\not \in T_n$ for any $j\not=1$.
		Observe that $(2,2)$ can also not be in $T_n$ as it is adjacent to $(3,1)$.
		But column $2$ only has the $1$s at $(2,2)$ and $(3,2)$, so $T_n$ contains no $1$s from column $2$ and $n$ and it has $n-1$ isolated $1$s from $n-2$ columns, which is a contradiction.
		(b)
		Suppose that $(3,n)\in T_n$. 
		Then $(3,j)\not \in T_n$ for any $j\not=n$.
		As $(3,2)\not \in T_n$, we must have the only available $1$ at $(2,2)$ from column $2$ in $T_n$.
		But then as $(2,2)$ is in a rectangle with $(1,3)$, we cannot have $(1,3)$ in $T_n$. As $(1,3)$ is the only $1$ in row $1$, $T_n$ has no $1$s from row $1$.
		But $T_n$ can also not have any $1$s from row $n$, as row $n$ only has a $1$ at $(n,n-1)$ which is adjacent to $(3,n)\in T_n$. Hence $T_n$ has $n-1$ isolated $1$s from $n-2$ rows, which is impossible. 
	Therefore, $i(\bmat{D}^{Q_n}_n)=n-2$.

%\ooo{Could I skip induction by assuming minimality on $n$ at the same time? }\\
\textbf{III.} 
%Note that $\bmat{D}_n$ for $n>4$ is not minimally non-superfirm, 
%Also not that the firmness of any submatrix of $\bmat{D}_n$ that is not a generalised binary matrix, so no $?$s in it, is firm by Gy\H{o}ri's Theorem \cite{Gyori:1984:Interval}. Sadly, this is also not much help here. 
We use induction on $n$. % that every proper submatrix of $\bmat{D}_n^{P}$ is firm for any  $P\subseteq Q_n$. 
For the base case take $n=4$ and observe that $\GG(\bmat{D}_4)$ has  the $5$-hole $\CC_4$ as an only odd hole and $\CC_4$ contains a vertex from each row and column of $\bmat{D}_4$. Therefore, any proper submatrix of $\bmat{D}_4^P$ is superfirm for any $P\subseteq Q_4$.
Assume that for $k<n$, all proper submatrices of $\bmat{D}_k^{P'}$ are firm for any $P'\subseteq Q_k$.
Let $P\subseteq Q_n$, and suppose that not every proper submatrix of $\bmat{D}_n^P$ is firm and let $\bmat{Y}$ be a smallest non-firm proper submatrix indexed by $I\times J$.
Note that we have $n\in I$ or $n\in J$, as otherwise $\bmat{Y}$ is a submatrix of $\bmat{D}_k^{P'}$ for some $k<n$ and $P'\subseteq \{(1,2),(2,1) \}$ and firm by either the induction hypothesis or by parts \textbf{I.} and \textbf{II.} of this proof as $P'\subsetneq Q_k$. 
By the minimality of $\bmat{Y}$ it must be mnf. 
So $\bmat{Y}$ has at least two non-zero entries in each row and column by Lemma \ref{lemma_mnf_at_least_two_1s}.
Hence $n\in I$ implies $n-1,n\in J$ and $n\in J$ implies $3,n\in I$. Thus we must have $3,n\in I$ and $n,n-1\in J$. 
Similarly,  if $i\in I$ for some $i>3$ then $i-1,i\in J$;  if $1\in I$ then $2,3\in J$ and if $1\in J$ then $2,3\in I$.

If $I=[n]$, then by the above we must have $J=[n]\setminus \{1 \}$. Then $(3,2)\cup S_n$ and $\{\{1,2,3 \}\times \{2,3\} \}\cup \RR_n $ with $S_n:=\{(i,i-1): i\in[4,n]:=\{4,\dots,n \}\}$ and $\RR_n:=\{ \{3,i  \}\times \{i-1,i  \} : i\in[4,n] \}$ give a feasible isolated set and rectangle cover of size $n-2$ of $\bmat{Y}$, hence we cannot have $I=[n]$.

So let $\ell$ be the largest row index of $\bmat{D}_n^P$ for which $\ell\not \in I$.
(a)
If $\ell=1$, then $I=[n]\setminus \{1  \}$. Then $[4,n]\subset I$ implies $[3,n]\subseteq J$, 
and $2\in I$ implies that column $1$ or $2$ are in $J$, so let $k\in J\cap \{1,2 \}$.
Then $(3,k)\cup S_n$ and 
$\{\{2,3\}  \times ( J\cap \{1,2,3\})\}\cup \RR_n$
 give a feasible isolated set and rectangle cover of size $n-2$ of $\bmat{Y}$, so $\ell\not =1$.

(b)
If $\ell=2$, then we have $1\not \in J$.
If $1\in I$, then $2,3\in J$ must hold, so we have $I=[n]\setminus\{2 \}$ and $J=[2,n]$. Then $(3,2)\cup S_n$ and $\{\{1,3 \}\times \{2,3\} \}\cup \RR_n $ give a feasible isolated set and rectangle cover of size $n-2$ of $\bmat{Y}$.
If $1\not \in I$, then $2\not \in J$, so we have $I=[3,n]$ and $J=[3,n]$. Then $S_n$ and $\RR$ give a feasible isolated set and rectangle cover of size $n-3$ of $\bmat{Y}$.

(c)
If $\ell>3$, then $(\ell+1,\ell)$ is a simplicial $1$ of $\bmat{Y}$ and its unique maximal rectangle is $\{3,\ell+1\} \times \{\ell,\ell+1 \}$. 
Remove this simplicial $1$ at $(\ell+1,\ell)$.
% we may assume that $(\ell+1,\ell)$ is in a maximum isolated set of $\bmat{Y}$ and its rectangle in a minimum cover, so let $\bmat{Y}'$ be obtained by removing $(\ell+1,\ell)$. 
But then $(\ell+2,\ell+1)$ becomes a simplicial $1$, so it can also be removed. We may repeat this process until at last $(n,n-1)$ becomes a simplicial $1$ and can be removed. Once $(n,n-1)$ is removed, column $n$ only consist of $0$s and a single $?$, hence can be dropped. Let the resulting matrix be $\bmat{Y}'$. As dropping a column which does not have any $1$s does not impact the isolation number and Boolean rank, by Lemma \ref{lemma_removing_simplicial_rectangle} $\bmat{Y}'$ satisfies $i(\bmat{Y}')+n-\ell =i(\bmat{Y})$ and $br(\bmat{Y}') +n-\ell=br(\bmat{Y})$. But then $\bmat{Y}'$ is just a proper submatrix of $\bmat{Y}$ formed by rows $(I\cap [\ell-1]) \times (J\cap [\ell-1])$, so firm. Hence $i(\bmat{Y})=br(\bmat{Y})$ which contradicts $\bmat{Y}$ being mnf. \hfill $\square$
\end{proof}

A proof which is very similar to the above may be applied to class $\bmat{T}_n$ to get our final lemma below and by this completing the proof of Theorem \ref{theorem_four_classes_of_mnf_matrices}.
\begin{lemma}
	For $n\ge 5$, $\bmat{T}_n^P$ is firm for all $P\subsetneq Q_n=\{(1,2),(2,1),(n,n)  \}$ and $\bmat{T}_n^{Q_n}$ is a mnf generalised binary matrix. In addition, $\bmat{T}_5$ is mnsf.
\end{lemma}
%Does anything else fit here?

\section{Conclusion}\label{section_conclusion}
In this paper, we studied firm and superfirm binary matrices. We showed that superfirmness is equivalent to having no odd holes in the rectangle cover graph. Then we presented four infinite classes of minimally non-firm binary matrices. 

We close with two future research directions. We suspect that every minimally non-superfirm matrix is firm and any minimally non-firm matrix $\bmat{X}\in\B^{m\times n}$ satisfies $|m-n|\le 1$.

%generalise mnf creation to automatic with stretching 1s next to exactly two vertices of odd hole

%simplicial 1s algo can factorise which matrices

%igraph's Mathematica interface, IGraph/M

\subsubsection{Acknowledgements}{I am very grateful to Ahmad Abdi for helping me begin studying firm matrices and for all the invaluable comments during our discussions.}
%
%% for Monson
%%	journal = {Bull. Inst. of Combin. App.},
%% 	publisher = {The Institute of Combinatorics and its Applications}


 \bibliographystyle{splncs04}
 \bibliography{refs_short}

%Author, F.: Article title. Journal \textbf{2}(5), 99--110 (2016)
%Author, F., Author, S.: Title of a proceedings paper. In: Editor,
%F., Editor, S. (eds.) CONFERENCE 2016, LNCS, vol. 9999, pp. 1--13.
%Springer, Heidelberg (2016). \doi{10.10007/123456789_0}
%Author, F., Author, S., Author, T.: Book title. 2nd edn. Publisher,
%Author, A.-B.: Contribution title. In: 9th International Proceedings
%on Proceedings, pp. 1--2. Publisher, Location (2010)
%LNCS Homepage, \url{http://www.springer.com/lncs}. Last accessed 4
%\end{thebibliography}
\end{document}